\documentclass[10pt]{article}
\usepackage{psfrag,amsmath}
\usepackage{amssymb}
\usepackage{mathrsfs}
\usepackage{amsfonts}
\usepackage{amsmath}
\usepackage{mathrsfs,color}
\usepackage{epstopdf}
\usepackage{graphicx}
\usepackage{mathrsfs,amscd,amssymb,amsthm,amsmath,bm,url}
\usepackage{tikz}
\usepackage{algorithm}
\usepackage[noend]{algpseudocode}
\usepackage{xcolor}
\usepackage{microtype}
\usetikzlibrary{positioning,backgrounds}
\usetikzlibrary{fadings}
\usetikzlibrary{patterns}
\usetikzlibrary{calc}
\usetikzlibrary{shadings}
\pgfdeclarelayer{background}
\pgfdeclarelayer{foreground}
\pgfsetlayers{background,main,foreground}

\makeatletter

\renewcommand{\@seccntformat}[1]{{\csname the#1\endcsname}{\normalsize .}\hspace{.5em}}
\makeatother

\def \[{\begin{equation}}
\def \]{\end{equation}}

\newtheorem{thm}{Theorem}[section]

\newtheorem{claim}{Claim}

\newtheorem{fact}{Fact}
\newtheorem{lem}[thm]{Lemma}

\newtheorem{ex}[thm]{Example}

\newenvironment{wst}
{\setlength{\leftmargini}{1.5\parindent}
 \begin{itemize}
 \setlength{\itemsep}{-1.1mm}}
{\end{itemize}}
\begin{document}

\setlength{\baselineskip}{0.20in}
\begin{center}{\Large \bf Sharp lower bounds and extremal graphs for the generalized
$k$-independence number\footnote{J.H.  financially supported by the National Natural Science Foundation of China (Grant No. 12001202), the Guangdong Basic and Applied Basic Research Foundation (Grant No. 2023A1515010838) and the Guangzhou Basic and Applied Basic Research Foundation (Grant No. 2024A04J3328).}}
\vspace{4mm}

{\large Jing Huang\footnote{Corresponding author. \\
\hspace*{5mm}{\it Email addresses}: jhuangmath@foxmail.com (J. Huang),\, 3059840843 (J. Tang)}, Jiaxin Tang\vspace{2mm}}

School of Mathematics and Information Science, Guangzhou University, Guangzhou, 510006, China
\end{center}


\begin{abstract}

A vertex set $S$ is a generalized $k$-independent set if the induced subgraph $G[S]$ contains no tree on $k$ vertices. The generalized $k$-independence number $\alpha_k(G)$ is the maximum size of such a set.
For a tree $T$ with $n$ vertices,
Bock et al. [J. Graph Theory 103 (2023) 661-673] and  Li et al. [Taiwanese J. Math. 27 (2023) 647-683] independently
showed that $\alpha_3(G)\geq \frac{2}{3}n$
and   identified the extremal trees that attain this lower bound. Subsequently,
Li and Zhou [Appl. Math. Comput. 484 (2025) 129018] established that  $\alpha_4(T) \geq \frac{3}{4}n$ and they further characterized all trees achieving this bound. This result was recently extended by Huang, who proved that $\alpha_4(G)\geq \frac{3}{4}(n-\omega(G))$ holds for every $n$-vertex graph,  where
$\omega(G)$ denotes the dimension of the cycle space of $G.$ The extremal graphs attaining this lower bound were also fully characterized. Based on these findings, Huang  proposed a conjecture concerning a lower bound for $\alpha_k(G)\ (k\geq2)$ together with  the corresponding extremal graphs, which naturally generalizes all the aforementioned results. In this paper, we confirm this conjecture here. We further quantify strict improvements over this bound when the equality conditions fail, and we provide a linear-time algorithm that constructs a generalized \(k\)-independent set of size at least \(\left\lceil\frac{k-1}{k}\left(n-\omega(G)\right)\right\rceil\).
\end{abstract}

\vspace{2mm} \noindent{\bf Keywords}: Generalized $k$-independence number; Dissociation set; Dimension of the cycle space.

\vspace{2mm} \noindent{\bf AMS Classification}:  05C69, 05C35

\setcounter{section}{0}

\section{\normalsize Introduction}\setcounter{equation}{0}
We begin with background and definitions. Our main results will also be given in this section.
\subsection{\normalsize Background and definitions}
Let $G=(V_G, E_G)$ be a finite simple  graph.
We call $|V_G|$ the \textit{order} of $G.$ The path, the star and the cycle of order $n$ are
denoted by  $P_n, S_n$ and  $C_n$, respectively.
For $v\in V_G,$ let $N_G(v)$ be the neighborhood of $v$ and  $d_G(v)=|N_G(v)|$ its \textit{degree}.
A vertex of degree one is a   \textit{pendant vertex}. A \textit{bridge} (also known as a \textit{cut-edge}) is an edge in a graph whose deletion  increases the number of connected components.
When  no confusion arises, we omit the subscript $G$. We follow the notation and terminology of  \cite{C95}.

Denote by $\omega(G)=|E_G|-|V_G|+c(G)$ the \textit{dimension of the cycle space} of $G$, where $c(G)$ is the number of connected components of $G$. A simple graph $G$ is called  \textit{acyclic} if it contains no cycles, whereas it is called an \textit{empty graph} if it has no edges.
For an induced subgraph $H$ of $G$, $G-H$ is the subgraph obtained from $G$ by deleting
all vertices of $H$ and all incident edges. For $W\subseteq V_G, G-W$ is the subgraph obtained from $G$ by deleting all
vertices in $W$ and all incident edges. For the sake of simplicity, we use $G-v, G-uv$ to denote the graph obtained from $G$ by deleting vertex $v \in V_G$, or edge $uv \in E_G$, respectively.
For two graphs $G_1$ and $G_2$, denote by $G_1 \cup G_2$ the  disjoint union of $G_1$ and $G_2$.
 For simplicity, we use $kG$ to denote the disjoint union of $k$ copies of $G$.

For an integer $k\geq 2$ and a subset $S\subseteq V_G$, we call $S$ a
\textit{generalized $k$-independent set} if the induced subgraph $G[S]$
contains no $k$-vertex tree. The \textit{generalized $k$-independence number} of
$G$, written as $\alpha_k(G)$, is the  maximum size of such a set.
The cases $k=2$ and $k=3$ recover the classical independence number and the dissociation number,
respectively (the latter introduced by
 Yannakakis \cite{Y81} and NP-complete to compute even on bipartite graphs).
Subsequent research by Cameron and Hell \cite{C06} showed that this problem can be solved in
polynomial time for specific graph families, including
chordal graphs, weakly chordal graphs, asteroidal triple-free graphs,
and interval-filament graphs, see also \cite{A07,B04,HB22,KK11,O11,T19} for complexity, approximability, and parameterization.

The problem concerning the bound of dissociation number in a given class of graphs is a
classical  problem and has been extensively studied. In 2009,  G\"{o}ring et al. \cite{GH09}
showed that
$$
\alpha_3(G)\geq\sum_{u\in V_G}\frac{1}{d_G(u)+1}+\sum_{uv\in E_G}\binom{|N_G[u]\cup N_G[v]|}{2}^{-1}.
$$
For a graph $G$ with $n$ vertices and
$m$ edges, Bre\v{s}ar et al. \cite{B11} pointed  out that $\alpha_3(G)\geq\frac{2n}{3}-\frac{m}{6}.$
Furthermore, they also proved that
\begin{equation*}
    \alpha_3(G)\geq
    \left\{
    \begin{array}{ll}
        \frac{n}{\left\lceil\frac{\Delta+1}{2}\right\rceil},& \textrm{if $G$ has maximum degree $\Delta,$}\\[8pt]
        \frac{4}{3}\sum_{u\in V_G}\frac{1}{d_G(u)+1},& \textrm{if $G$ has no isolated vertex,}\\[5pt]
        \frac{n}{2},& \textrm{if $G$ is outerplanar,}\\[5pt]
        \frac{2n}{3},& \textrm{if $G$ is a tree.}
    \end{array}
    \right.
\end{equation*}
Bre\v{s}ar et al. \cite{BJ13}  further  demonstrated that $\alpha_3(G)\geq\frac{2n}{k+2}-\frac{m}{(k+1)(k+2)},$ where
$k=\left\lceil\frac{m}{n}\right\rceil-1.$
Li and Sun \cite{LS23}  proved that
$\alpha_3(F)\geq \frac{2n}{3}$ for each
acyclic graph $F$ with order $n$. They also characterized all the corresponding
extremal acyclic graphs.
Bock et al. \cite{BP23a} generalized the result and  showed that if $G$ is a graph with
$n$ vertices, $m$ edges, $k$ components, and $c_1$ induced cycles of length 1 modulo 3, then
$\alpha_3(G)\geq n-\frac{1}{3}(m+k+c_1)$, the extremal graphs in which every two cycles are
vertex-disjoint were identified. In another paper, they \cite{BP23b}  provided several upper bounds on the dissociation number by utilizing  independence number in some specific classes of graphs, including bipartite graphs,  triangle-free graphs and subcubic graphs. The extremal graphs that reach the partial bounds were
characterized. Li and Zhou \cite{LZ25} established  $\alpha_4(T) \geq \frac{3}{4}n$ for a tree with $n$ vertices  and they further characterized all trees achieving this bound.
Recently, Huang \cite{H25} extended this result by proving that $\alpha_4(G)\geq \frac{3}{4}(n-\omega(G))$ holds for every $n$-vertex graph, and fully characterized the extremal graphs attaining this lower bound by using a novel approach. In the same paper,
 Huang also proposed a conjecture   regarding a lower bound for the generalized $k$-independence number and the corresponding extremal graphs for every integer $k\geq2$ (Theorem \ref{thm1.3}). In this paper, we confirm this conjecture. Beyond the sharp bound and the full extremal characterization, we derive several refined lower bounds capturing violations of the equality conditions (Remark 4.6) and give a linear-time constructive algorithm achieving \(\left\lceil\frac{k-1}{k}\left(n-\omega(G)\right)\right\rceil\)

\subsection{\normalsize Main results}

In this subsection, we give some basic notation and then describe our main result.
If a graph has order less than $k$, its generalized $k$-independence number is trivially equal to its order. Therefore, throughout this paper, we assume by default that all graphs under consideration have order at least $k$.

For integers \( k \geq 2 \) and \( i \geq 1 \), let \( R_i^k \) be the family of trees on \( ik \) vertices defined recursively as follows:
\( R_1^k \) is the set of all trees on \( k \) vertices; for \( i \geq 2 \), a tree \( T \) lies in \( R_i^k \) if there exists \( T' \in R_{i-1}^k \) and a tree \( T''\) on \( k \) vertices such that \( T \) is obtained from \( T' \cup T''\) by adding a single edge joining a vertex of \( T' \) to a vertex of \( T'' \). Equivalently, \( T \in R_i^k\) if and only if  \( T \) can be obtained by attaching \( i \) edge-disjoint \( k \)-vertex trees one by one along single edges.

Our first main result establishs a lower bound on the generalized $k$-independence number of a tree with fixed order, and all the corresponding extremal trees are characterized, as stated below.

\begin{thm}\label{thm1.1}
Let $T$ be a tree on $n$ vertices. Then $\alpha_k(T)\geq \frac{k-1}{k}n$ for every integer $k\geq2.$  Equality holds if and only if $ k \mid n $  and $T\in R_{n/k}^k.$
\end{thm}

If $G$ is an $n$-vertex disconnected  graph with
$G=\bigcup_{i=1}^lG_i$, where $G_i$ is a  component of $G$ with order $n_i\ (1\leq i\leq l)$, then it is obvious that
$$
\sum_{i=1}^ln_i=n,\ \ \sum_{i=1}^l\alpha_k(G_i)=\alpha_k(G),
$$
and thus the following result is a direct consequence of Theorem \ref{thm1.1}.

\begin{thm}\label{thm1.2}
Let $F$ be  an acyclic graph with $n$ vertices. Then $\alpha_k(F)\geq \frac{k-1}{k}n$ for every integer $k\geq2.$  Equality holds if and only if each connected component, say $T,$ of $F$ satisfies $ k \mid |V_T|$ and
$ T\in R_{|V_T|/k}^k.$
\end{thm}

Let $G$ be a graph with pairwise vertex-disjoint cycles, and let $\mathscr{C}_G$ denote the set of all cycles in $G$. By contracting  each cycle of $G$  to a single vertex  we obtain an acyclic graph $T_G$. More precisely, let $U_G$ be the set of  vertices not on any cycle and $W_{\mathscr{C}_G}=\{v_C:C\in\mathscr{C}_G \}$ the contracted cycle-vertices; then $V_{T_G}=U_G\cup W_{\mathscr{C}_G},$ and adjacency is defined as follows: two vertices in $U_G$ are adjacent in $T_G$ if and only if they are adjacent in $G$, a vertex $u\in U_G$ is adjacent to a vertex $v_C\in W_{\mathscr{C}_G}$ if and only if $u$ is adjacent (in $G$) to a vertex in the cycle $C$, and two vertices $v_{C^1}, v_{C^2}$ in $W_{\mathscr{C}_G}$ are adjacent in $T_G$ if and only if there exists an edge in $G$ joining a vertex of $C^1\in \mathscr{C}_G$ to a vertex of $C^2\in \mathscr{C}_G.$ The graph $T_G-W_{\mathscr{C}_G}$ equals
$\Gamma_G,$ obtained from $G$ by deleting all cycle vertices and incident edges. Figure \ref{fig1}
gives an example for  $G, T_G$ and $\Gamma_G.$

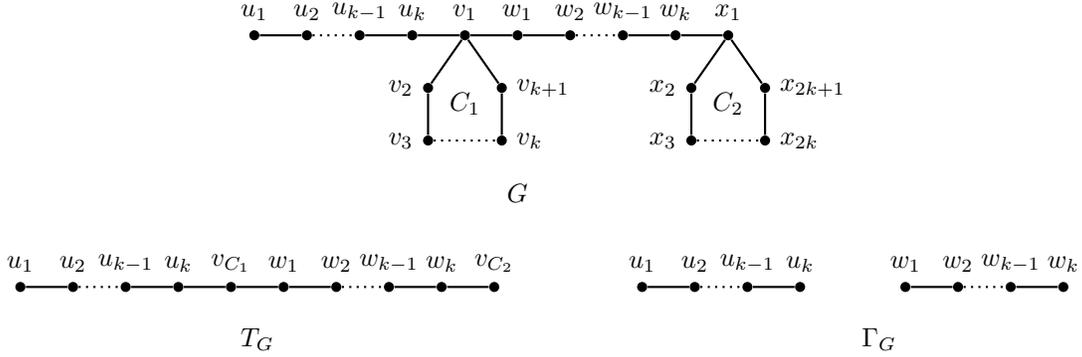
\begin{figure}[!ht]
\centering
  \begin{tikzpicture}[scale = 0.7]
  \tikzstyle{vertex}=[circle,fill=black,minimum size=0.38em,inner sep=0pt]
  \node[vertex] (u_1) at (0,0)[label=above:$u_1$]{};
  \node[vertex] (u_2) at (1,0)[label=above:$u_2$]{};
  \node[vertex] (u_k-1) at (2,0)[label=above:$u_{k-1}$]{};
  \node[vertex] (u_k) at (3,0)[label=above:$u_k$]{};
  \node[vertex] (v_1) at (4,0)[label=above:$v_1$]{};
  \node[vertex] (v_2) at (3.3,-1)[label=left:$v_2$]{};
  \node[vertex] (v_3) at (3.3,-2)[label=left:$v_3$]{};
   \node[vertex] (v_k) at (4.7,-2)[label=right:$v_k$]{};
  \node[vertex] (v_k+1) at (4.7,-1)[label=right:$v_{k+1}$]{};
  \node[vertex] (w_1) at (5,0)[label=above:$w_1$]{};
  \node[vertex] (w_2) at (6,0)[label=above:$w_2$]{};
  \node[vertex] (w_k-1) at (7,0)[label=above:$w_{k-1}$]{};
  \node[vertex] (w_k) at (8,0)[label=above:$w_k$]{};
  \node[vertex] (x_1) at (9,0)[label=above:$x_1$]{};
  \node[vertex] (x_2) at (8.3,-1)[label=left:$x_2$]{};
  \node[vertex] (x_3) at (8.3,-2)[label=left:$x_3$]{};
  \node[vertex] (x_2k) at (9.7,-2)[label=right:$x_{2k}$]{};
  \node[vertex] (x_2k+1) at (9.7,-1)[label=right:$x_{2k+1}$]{};
 \node at (4,-1.3) {$C_1$};
 \node at (9,-1.3) {$C_2$};
  \draw[thick] (u_1)--(u_2);
  \draw[thick] (u_k-1)--(u_k)--(v_1)--(v_2)--(v_3);
  \draw[thick] (v_k)--(v_k+1)--(v_1)--(w_1)--(w_2);
  \draw[thick] (w_k-1)--(w_k)--(x_1)--(x_2)--(x_3);
  \draw[thick] (x_2k)--(x_2k+1)--(x_1);
  \draw[thick, dotted]   (u_2)--(u_k-1);
  \draw[thick, dotted]   (v_3)--(v_k);
  \draw[thick, dotted]   (w_2)--(w_k-1);
  \draw[thick, dotted]   (x_3)--(x_2k);
  \draw (5,-3)node{$G$};
  \end{tikzpicture}\\[12pt]
  \begin{tikzpicture}[scale = 0.7]
  \tikzstyle{vertex}=[circle,fill=black,minimum size=0.38em,inner sep=0pt]
  \node[vertex] (u_1) at (0,0)[label=above:$u_1$]{};
  \node[vertex] (u_2) at (1,0)[label=above:$u_2$]{};
  \node[vertex] (u_k-1) at (2,0)[label=above:$u_{k-1}$]{};
  \node[vertex] (u_k) at (3,0)[label=above:$u_k$]{};
  \node[vertex] (v_1) at (4,0)[label=above:$v_{C_1}$]{};
  \node[vertex] (w_1) at (5,0)[label=above:$w_1$]{};
  \node[vertex] (w_2) at (6,0)[label=above:$w_2$]{};
  \node[vertex] (w_k-1) at (7,0)[label=above:$w_{k-1}$]{};
  \node[vertex] (w_k) at (8,0)[label=above:$w_k$]{};
  \node[vertex] (x_1) at (9,0)[label=above:$v_{C_2}$]{};
  \draw[thick] (u_1)--(u_2);
  \draw[thick] (u_k-1)--(u_k)--(v_1)--(w_1)--(w_2);
  \draw[thick] (w_k-1)--(w_k)--(x_1);
  \draw[thick, dotted]   (u_2)--(u_k-1);
  \draw[thick, dotted]   (w_2)--(w_k-1);
  \draw (4.5,-1)node{$T_G$};
   \end{tikzpicture}
  \hspace{3em}
 \begin{tikzpicture}[scale = 0.7]
  \tikzstyle{vertex}=[circle,fill=black,minimum size=0.38em,inner sep=0pt]
  \node[vertex] (u_1) at (0,0)[label=above:$u_1$]{};
  \node[vertex] (u_2) at (1,0)[label=above:$u_2$]{};
  \node[vertex] (u_k-1) at (2,0)[label=above:$u_{k-1}$]{};
  \node[vertex] (u_k) at (3,0)[label=above:$u_k$]{};
  \node[vertex] (w_1) at (5,0)[label=above:$w_1$]{};
  \node[vertex] (w_2) at (6,0)[label=above:$w_2$]{};
  \node[vertex] (w_k-1) at (7,0)[label=above:$w_{k-1}$]{};
  \node[vertex] (w_k) at (8,0)[label=above:$w_k$]{};
  \draw[thick] (u_1)--(u_2);
  \draw[thick] (u_k-1)--(u_k);
  \draw[thick] (w_1)--(w_2);
  \draw[thick] (w_k-1)--(w_k);
  \draw[thick, dotted]   (u_2)--(u_k-1);
  \draw[thick, dotted]   (w_2)--(w_k-1);
  \draw (4.5,-1)node{$\Gamma_G$};
   \end{tikzpicture}
  \caption{An example for  $G, T_G$ and $\Gamma_G$}\label{fig1}
\end{figure}

Our second main result extends Theorem \ref{thm1.2} to general graphs, which establishes a sharp lower bound on the generalized $k$-independence number and characterizes all the corresponding extremal graphs for every integer $k\geq2$.
\begin{thm}\label{thm1.3}
Let $G$ be an $n$-vertex  graph with the dimension of cycle space $\omega(G)$. Then
\begin{equation}\label{eq1.1}
\alpha_k(G)\geq \frac{k-1}{k}[n-\omega(G)]
\end{equation}
for every integer $k\geq2.$ The equality holds if and only if all the following conditions hold for $G$
\begin{wst}
\item[{\rm (i)}] the cycles (if any) of $G$ are pairwise vertex-disjoint;
\item[{\rm (ii)}] each cycle (if any) of $G$ has length 1 modulo $k$;
\item[{\rm (iii)}] each connected component, say $T,$  of $\Gamma_G$ satisfies $ k \mid |V_T|$ and
$ T\in R_{|V_T|/k}^k.$
\end{wst}
\end{thm}

\begin{ex}
If $G$ is the graph as depicted in Figure \ref{fig1}, then $G$ satisfies Theorem \ref{thm1.3} (i)-(iii) and $\alpha_k(G)=\frac{k-1}{k}[n-\omega(G)]$ holds, where $\alpha_k(G)=5(k-1), n=5k+2$ and $\omega(G)=2$.
\end{ex}
 The paper is organized as follows: Section 2 collects preliminaries; Section 3 proves Theorem 1.1; Section 4 proves Theorem 1.3 and presents quantitative refinements; Appendix A gives a linear-time algorithm.

\section{\normalsize Preliminary results}\setcounter{equation}{0}\label{sec2}
In this section, we present some initial findings that will serve as the basis for proving our main results.
 The following  result
immediately follows from the definition of the  generalized $k$-independence number.
\begin{lem}\label{lem2.1}
Let $G=(V_G,E_G)$ be a simple graph. Then
\begin{wst}
\item[{\rm (i)}] $\alpha_k(G)-1\leq\alpha_k(G-v)\leq\alpha_k(G)$ for any $v\in V_G;$
\vspace{1mm}
\item[{\rm (ii)}] $\alpha_k(G-e)\geq \alpha_k(G)$ for any $e\in E_G$.
\end{wst}
\end{lem}

\begin{lem}\label{lem2.2}
Let $P_n$ and $C_n$ be a path and a cycle on $n$ vertices, respectively. Then
\begin{wst}
\item[{\rm (i)}]$\alpha_k(P_n)=\left\lceil\frac{k-1}{k}n\right\rceil;$
\vspace{1mm}
\item[{\rm (ii)}]$\alpha_k(C_n)=\left\lfloor\frac{k-1}{k}n\right\rfloor.$
\end{wst}
\end{lem}
\begin{proof}
Write \( n = kq + r \) with integers \( q \geq 1 \) and \( 0 \leq r \leq k - 1 \).

(i)\ On the one hand,  partition $V_{P_n}$ into \( q \) disjoint blocks of \( k \) consecutive vertices and one final block of size \( r \) (possibly \( r = 0 \)). In each full \( k \)-block, a generalized \( k \)-independent set can contain at most \( k - 1 \) vertices, otherwise it would contain \( k \) consecutive chosen vertices and hence a \( P_k \). The final block of size \( r \) contributes at most \( r \) vertices. Hence
\begin{eqnarray}\label{eq2.1}
\alpha_k(P_n) \leq q(k - 1) + r = n - q = n - \left\lfloor \frac{n}{k} \right\rfloor = \left\lceil \frac{k - 1}{k} n \right\rceil.
\end{eqnarray}

On the other hand, assume $P_n=u_1u_2 \cdots u_n $ and choose \(
S^* = \{u_i : i \notin \{k, 2k, 3k, \ldots, qk\}\},
\) i.e., delete every \( k \)-th vertex along the path. Then between any two deleted vertices there are exactly \( k - 1 \) consecutive kept vertices, and the terminal segment (if \( r > 0 \)) has length \( r < k \). Thus there is no run of \( k \) consecutive chosen vertices, so \( S^* \) is a generalized \( k \)-independent set of $P_n$. This  gives
\begin{eqnarray}\label{eq2.2}
\alpha_k(P_n) \geq|S^*| = n - q = \left\lceil \frac{k-1}{k} n \right\rceil.
\end{eqnarray}
Therefore, $\alpha_k(P_n) =\left\lceil \frac{k-1}{k} n \right\rceil$ by \eqref{eq2.1}-\eqref{eq2.2}.

\vspace{1mm}

(ii)\ Assume that $C_n=v_1v_2\cdots v_nv_1$. Consider the \( n \) cyclic $k$-windows $
W_i = \{v_i, v_{i+1}, \ldots, v_{i+k-1}\}\, (1\leq i\leq n),$ where the subscripts are  modulo $n$.
Let \( S \) be a generalized \( k \)-independent set of $C_n$. Consider the set
\begin{eqnarray*}
\mathcal{A} := \left\{(i,x) : i \in \{1,\ldots,n\}, \, x \in W_i \cap (V_{C_n} \setminus S)\right\}.
\end{eqnarray*}
Note that every window \( W_i \)  contains at least one vertex not in \( S \); otherwise \( W_i \subseteq S \) would be a \( P_k \). Each window \(W_i\) contributes at least one ordered pair to \(\mathcal{A}\). Therefore,
\begin{eqnarray}\label{eq2.1a}
|\mathcal{A}| \geq n.
\end{eqnarray}
Fix a vertex \( x = v_j \in V_{C_n} \setminus S \). The windows that contain \( x \) are precisely
\(
W_{j-k+1}, W_{j-k+2}, \ldots, W_{j-1}, W_j,
\)
\( k \) windows in total. Thus \( x \) appears in exactly \( k \) ordered pairs of \( \mathcal{A} \). Summing over all \( x \in V_{C_n} \setminus S \) gives
\begin{eqnarray}\label{eq2.1b}
|\mathcal{A}| = \sum_{x \in V_{C_n} \setminus S} k = k \, |V_{C_n} \setminus S|.
\end{eqnarray}
Combining  \eqref{eq2.1a}-\eqref{eq2.1b} gives $n \leq k |V_{C_n} \setminus S|.$ Therefore,
\begin{eqnarray}\label{eq2.3}
 \alpha_k(C_n)=|S| \leq n - \left\lceil \frac{n}{k} \right\rceil = \left\lfloor \frac{k-1}{k} n \right\rfloor.
\end{eqnarray}

If \( r = 0 \) (i.e., \( k \mid n \)), then delete every \( k \)-th vertex around the cycle. The remaining vertices come in runs of length exactly \( k - 1 \), so no \( P_k \) appears. This means they form a generalized $k$-independent set in $C_n$ and the set has size \( n - n/k = \left\lfloor \frac{k-1}{k} n \right\rfloor \).  This together with \eqref{eq2.3} implies that $\alpha_k(C_n)=\left\lfloor \frac{k-1}{k} n \right\rfloor.$

If \( 1 \leq r \leq k - 1 \), then delete every \( k \)-th vertex and one additional vertex, say \( v_1 \). Then the consecutive runs of remaining  vertices between deletions have lengths all less than \( k \). Hence no \( P_k \) occurs, and the set has size $
n - (q + 1) = (kq + r) - (q + 1) = (k - 1)q + (r - 1) = \left\lfloor \frac{k - 1}{k} n \right\rfloor.$
By a similar argument, we obtain $\alpha_k(C_n)=\left\lfloor \frac{k-1}{k} n \right\rfloor.$

\end{proof}
From the definition of the dimension of  cycle space, the following conclusion is straightforward.
For completeness, a detailed proof is provided herein.
\begin{lem}\label{lem2.3}
Let $G$ be a graph with $x\in V_G$.
\begin{wst}
\item[{\rm (i)}] If $x$ lies on no cycle of $G$, then $\omega(G)=\omega(G-x);$
\item[{\rm (ii)}] If $x$ lies on a cycle of $G$, then $\omega(G-x)\leq \omega(G)-1;$
\item[{\rm (iii)}] If the cycles of $G$ are pairwise vertex-disjoint, then $\omega(G)$ precisely equals the number of cycles in $G$.
\end{wst}
\end{lem}
\begin{proof}

(i)\  Let $d=d_G(x)$. If $x$ lies on no cycle, then all edges incident with $x$ are bridges.
Deleting $x$ removes $d$ bridges and one vertex, and splits its original connected component into exactly $d$ connected components; thus
$c(G-x)=c(G)-1+d$. With $|V_{G-x}|=|V_G|-1$ and $|E_{G-x}|=|E_G|-d$, we get
\begin{equation*}
\omega(G-x)=(|E_G|-d)-(|V_G|-1)+(c(G)-1+d)=\omega(G).
\end{equation*}

(ii)\ If $x$ lies on a cycle, then the connected component containing $x$ can split into at most $d-1$ additional connected components, so
$c(G-x)\leq c(G)+d-2$. Therefore
\begin{equation*}
\omega(G-x)= (|E_G|-d) - (|V_G|-1) + c(G-x)
\leq |E_G|-|V_G|+c(G)-1 = \omega(G)-1,
\end{equation*}

(iii)\  If the cycles of $G$ are pairwise vertex-disjoint, then
remove, one by one, a single edge from each cycle in $\mathscr{C}_G$.
At each step, the removed edge lies on a cycle and hence is not a bridge, so the number of connected  components does not change and the dimension of  cycle space of the current graph decreases by exactly $1$.
Because cycles are vertex-disjoint, removing an edge from one cycle does not affect the others, and after $|\mathscr{C}_G|$ steps the resulting graph is acyclic, whence its cycle-space dimension is $0$.
Therefore $\omega(G)=|\mathscr{C}_G|$.
\end{proof}

\section{\normalsize Proof of  Theorem \ref{thm1.1}}\setcounter{equation}{0}\label{sec3}
In this section, we prove Theorem \ref{thm1.1}, which determines a sharp lower bound on the generalized
$k$-independence number among  all trees of a fixed order, and we also characterize the  extremal graphs that
attain this bound.

Root $T$ at a vertex $r;$ the unique path $rTv$ from $r$ to
$v$ defines the \textit{level} of $v$ as its length. A vertex on  $rTv$ other
than  $v$ is  an \textit{ancestor} of
$v$;  each vertex with $v$ as its ancestor is a \textit{descendant} of $v$.
The immediate ancestor of $v$ is its \textit{parent}, and the vertices whose parent is
$v$ are its \textit{children}.

\vspace{2mm}
\noindent{\bf Proof of Theorem \ref{thm1.1}.}\ Observe that a subset \( S \subseteq V_T \)
is a generalized \( k \)-independent set of $T$ if and only if every connected component of
\( T[S] \) has at most \( k - 1 \) vertices. Therefore, if we let
\begin{equation*}
\tau_k(T) := \min\{|X| : \hbox{ every connected component of } T - X \hbox{ has at most  } k - 1 \hbox{ vertices}\},
\end{equation*}
then
\begin{eqnarray}\label{3.1}
\alpha_k(T) = n - \tau_k(T).
\end{eqnarray}

\begin{claim}
For every $k\geq2$, there exists a subset \( X \subseteq V_T \) with
\(
|X| \leq \left\lfloor \frac{n}{k} \right\rfloor
\)
such that every connected component of \( T - X \) has at most \( k - 1 \) vertices.

\end{claim}
\begin{proof}
Root \( T \) arbitrarily. For a vertex \( v \), let  $T^v$ be a subtree
of $T$ rooted at $v$. While $T$ has at least \( k \) vertices, choose a vertex \( v\in V_{T} \) such that \( |V_{T^v}| \geq k \), and subject to this condition,
the level of $v$ is as large as possible. By choice, every descendant \( u \) of \( v \) satisfies \( |V_{T^u}| \leq k-1 \). Put \( v \) into \( X \) and delete \( v \). Then every former child subtree becomes a connected component of order at most \( k-1 \) and will never need further deletions.

Let the chosen vertices be \( v_{1},\ldots,v_{t} \) in order, and write \( s_{i}:=|V_{T^{v_i}}| \) at the moment \( v_{i} \) is chosen. Each step removes exactly \( s_{i} \) vertices from the current tree (the vertex \( v_{i} \) plus all vertices in its child subtrees), so after \( t \) steps fewer than \( k \) vertices remain. Hence,
\begin{eqnarray*}
n = (s_{1}+\cdots+s_{t})+r \quad \text{with } s_{i} \geq k \text{ and } 0 \leq r < k,
\end{eqnarray*}
(the remainder $r<k$ vertices form the final connected component). Therefore, \( tk \leq s_{1}+\cdots+s_{t} \leq n \), i.e. \( t \leq \lfloor n/k \rfloor \). Taking \( X = \{v_{1},\ldots,v_{t}\} \) gives the claim. and then (1) yields the bound on \( \alpha_{k}(T) \).
\end{proof}

Claim 1 and \eqref{3.1} yield  \( \alpha_{k}(T)\geq n-\left\lfloor \frac{n}{k} \right\rfloor\geq \frac{k-1}{k}n \). In the following, we show  $\alpha_{k}(T)=\frac{k-1}{k}n $ if and only if $T\in R_{n/k}^k.$

\vspace{1mm}
For ``necessity", assume $\alpha_k(T) = \frac{k-1}{k} n$, necessarily $k\mid n$. Then
\(
n - \alpha_k(T) = \tau_k(T) = \frac{n}{k}.
\)
Run the proof  from Claim 1. Recall the notation $s_1, \ldots, s_t$ (orders of the removed rooted subtrees) and $r < k$ with
\begin{eqnarray*}
tk =n= (s_1 + \cdots + s_t) + r, \quad s_i \geq k.
\end{eqnarray*}
The identity then forces
\(s_1 = \cdots = s_t = k\) and  \( r = 0. \)

That is to say, at each step there is a vertex $v$ whose rooted subtree has exactly $k$ vertices; removing that single vertex detaches a  $k$-vertex tree (its whole subtree) from the remainder.

Therefore, iterating the deletions decomposes $T$ into $t=\frac{n}{k}$ vertex-disjoint subtrees of order $k$, each joined to the previously remaining part by a single edge (the edge from the removed vertex to its parent). Reversing this process constructs $T$ from a single $k$-vertex tree by repeatedly attaching a $k$-vertex tree with a single edge.  Consequently, $T\in R^{k}_{n/k}$.

\vspace{1mm}
For ``sufficiency", suppose   $T\in R_{n/k}^{k}$. Fix the natural ``$k$-block tree" order in which $T$ is obtained: at each step a new $k$-vertex tree (a ``$k$-block") is attached to the current graph by a single edge. For each non-root $k$-block, delete the endpoint inside that new $k$-block of its unique attaching edge; additionally, delete any one vertex in the root $k$-block. Exactly one vertex per $k$-block is deleted, so altogether $n/k$ vertices are deleted from $T$. Because every inter $k$-block edge has an endpoint deleted (the endpoint in the child $k$-block), all inter $k$-block edges disappear, what remains is a  forest in which every component has size $k-1$. Hence,
\(
\alpha_{k}(T) \geq n - \frac{n}{k} = \frac{k-1}{k}n.
\)
On the other hand, since $T$ contains $n/k$ vertex-disjoint $k$-vertex trees (the $k$-blocks), any generalized $k$-independent set must delete at least one vertex from each $k$-block; therefore $\alpha_{k}(T) \leq n - n/k$. Thus $\alpha_{k}(T) = \frac{k-1}{k}n$.

 \qed

\section{\normalsize Proof of  Theorem \ref{thm1.3}}\setcounter{equation}{0}\label{sec4}
In this section, we give  a proof for Theorem \ref{thm1.3}, which
establishes a sharp lower bound on the generalized $k$-independence number of a general graph. The corresponding extremal graphs are also characterized. More specifically, we give the proof  according to the following steps. We first show that the inequality in \eqref{eq1.1} holds. Then we present a few technical lemmas. Finally, we characterize all the graphs which attain the equality in \eqref{eq1.1}.
Without loss of generality,  assume that $G$ is  connected.
\vspace{2mm}

\begin{lem}\label{lem3.1}
The inequality \eqref{eq1.1} holds.
\end{lem}

\begin{proof}
We show \eqref{eq1.1} holds by induction on $\omega(G).$  If $\omega(G)=0$, then $G$ is a tree and the result follows immediately by Theorem \ref{thm1.1}.  Now assume that $\omega(G) \geq 1,$  i.e., $G$ has at least one cycle. Let $x$ be a vertex lying on some cycle.  By Lemmas \ref{lem2.1}(i) and \ref{lem2.3}(ii), we have
\begin{eqnarray}\label{eq3.1}
\alpha_k(G)\geq \alpha_k(G-x),\ \ \omega(G)\geq \omega(G-x)+1.
\end{eqnarray}
Applying the induction hypothesis yielding
\begin{eqnarray}\label{eq3.2}
\alpha_k(G-x)\geq\frac{k-1}{k}[n-1-\omega(G-x)].
\end{eqnarray}
Therefore,  \eqref{eq3.1}-\eqref{eq3.2} lead to
\begin{eqnarray*}
\alpha_k(G)\geq \frac{k-1}{k}[n-\omega(G)],
\end{eqnarray*}
as desired.
\end{proof}

For convenience, a graph is called \textit{$k$-good} if it achieves equality in \eqref{eq1.1}. Note that if $G$ is $k$-good, then $k\mid (n-\omega(G))$. In the following, we aim to provide some fundamental characterizations of $k$-good graphs. The following result is a  direct consequence of Lemma \ref{lem3.1}.

\begin{lem}\label{lem3.2}
A disconnected graph is $k$-good if and only if each connected component of it is $k$-good.
\end{lem}

\begin{lem}\label{lem3.3}
Let $G$ be a graph with $x\in V_G$  lying on some cycle of $G$. If $G$ is $k$-good, then
\begin{wst}
\item[{\rm (i)}] $\alpha_k(G)=\alpha_k(G-x);$
\item[{\rm (ii)}] $\omega(G)=\omega(G-x)+1;$
\item[{\rm (iii)}] $G-x$ is $k$-good.
\end{wst}
\end{lem}
\begin{proof}
The fact that $G$ is  $k$-good together with the proof of Lemma \ref{lem3.1} forces all equalities in \eqref{eq3.1}-\eqref{eq3.2}. Therefore, (i)-(iii) are all derived.

\end{proof}

An induced cycle $H$  of a graph $G$ is called a \textit{pendant cycle} if $H$ contains a unique vertex of degree 3 and each of its other  vertices is of degree 2 in $G$. For example, the graph $G$ as depicted in Figure \ref{fig1}  has exactly one pendant cycle $C_2.$

\begin{lem}\label{lem3.4}
Let $G$ be a graph with $C_q$ being a pendant cycle of $G$. Let $H:=G-C_q$. If $G$ is $k$-good, then
\begin{wst}
\item[{\rm (i)}] $q\equiv 1\pmod{k};$
\item[{\rm (ii)}] $\alpha_k(G)=\alpha_k(H)+\frac{k-1}{k}(q-1);$
\item[{\rm (iii)}] $H$ is $k$-good.
\end{wst}
\end{lem}
\begin{proof}
Let $x$ be the unique vertex of degree $3$ on $C_q.$ Then $G-x=P_{q-1}\cup H$.
By Lemmas \ref{lem3.2} and \ref{lem3.3}(iii), we obtain both $P_{q-1}$ and $H$ are good. This means $\alpha_k(P_{q-1})=\frac{k-1}{k}(q-1)$. Hence, by Lemma \ref{lem2.2}(i), we have $q\equiv 1\pmod{k}$.
 Applying
Lemma \ref{lem3.3}(i), one has
$$
\alpha_k(G)=\alpha_k(G-x)=\alpha_k(H)+\alpha_k(P_{q-1})
=\alpha_k(H)+\frac{k-1}{k}(q-1).
$$

\end{proof}

\begin{lem}\label{lem3.6}
If $G$ is $k$-good, then
\begin{wst}
\item[{\rm (i)}] the cycles (if any) of $G$ are pairwise vertex-disjoint;
\item[{\rm (ii)}] each cycle (if any) of $G$ has length 1 modulo $k$;
\item[{\rm (iii)}] $\alpha_k(G)=\alpha_k(\Gamma_G)+\sum_{C\in \mathscr{C}_G}\frac{k-1}{k}(|V_C|-1).$
\end{wst}
\end{lem}
\begin{proof}
If $G$ is not cycle-disjoint, then choose a vertex $x$ lying on at least two cycle such that $d_G(x)\geq3.$
Thus two distinct pairs of neighbors of $x$ remain connected in $G-x, $ implying  $c(G-x)\leq d_G(x)-2.$
This together with the fact that $|V_{G-x}|=|V_G|-1$ and $|E_{G-x}|=|E_G|-d_G(x)$ yields $\omega(G-x)\leq \omega(G)-2,$ which contradicts to Lemma \ref{lem3.3}(ii).
This completes the proof of (i).

We proceed by induction on the order $n$ of $G$ to prove (ii) and (iii). Since $G$ is $k$-good, it must be $k \mid (n-\omega(G))$. If $n=k,$ then $\omega(G)=0,$ i.e., $G$ is a $k$-vertex tree.
Thus, (ii)-(iii) establish obviously.
Suppose that (ii) and (iii) hold for any $k$-good graph of order smaller than $n$, and suppose $G$ is a $k$-good graph with order $n\geq k+1.$

Since Lemma \ref{lem3.6}(i) shows cycles are disjoint for $k$-good graphs, $T_G$ is well-defined. If $T_G$ is an empty graph, then $G\cong C_n.$ Thus (ii) and (iii)  follow from  the following fact, which follows directly from Lemma \ref{lem2.2}(ii).
\begin{fact}\label{fact1}
$C_n$ is $k$-good $\Leftrightarrow$ $n\equiv 1\pmod{k}$ $\Leftrightarrow$
$\alpha_k(C_n)=\frac{k-1}{k}(n-1)$.
\end{fact}

If $T_G$ contains at least one edge, then $\mathcal{P}(T_G)\neq\emptyset$, where $\mathcal{P}(T_G)$ denotes the set of all pendant vertices of $T_G$. In order to complete the proof of (ii) and (iii) in this case, it suffices to consider the following two possible cases.

{\bf{Case 1.}}\ $\mathcal{P}(T_G)\cap  W_{\mathscr{C}_G}\neq\emptyset$. In this case, $G$ has a pendant cycle, say $C_q$. Let $H_1=G-C_q$. It follows from Lemma \ref{lem3.4}(iii) that $H_1$ is $k$-good. Applying the induction hypothesis to $H_1$ yields
\begin{wst}
\item[{\rm (a)}] each cycle (if any) of $H_1$ has length 1 modulo $k$;
\vspace{1mm}
\item[{\rm (b)}] $\alpha_k(H_1)=\alpha_k(\Gamma_{H_1})+\sum_{C\in \mathscr{C}_{H_1}}\frac{k-1}{k}(|V_C|-1).$
\end{wst}
Assertion (a) and Lemma \ref{lem3.4}(i) imply that each cycle (if any) of $G$ has length 1 modulo $k$ since $\mathscr{C}_G=\mathscr{C}_{H_1}\cup\{C_q\}.$ Thus, (ii) holds in this case.

Combining with Lemma \ref{lem3.4}(ii) and Assertion (b) we have
\begin{eqnarray}\label{eq3.5}
\alpha_k(G)=\alpha_k(H_1)+\frac{k-1}{k}(q-1)=\alpha_k(\Gamma_{H_1})
+\sum_{C\in \mathscr{C}_{H_1}}\frac{k-1}{k}(|V_C|-1)+\frac{k-1}{k}(q-1).
\end{eqnarray}
Note that $\Gamma_G\cong \Gamma_{H_1}$ and
$$
\sum_{C\in \mathscr{C}_{H_1}}\frac{k-1}{k}(|V_C|-1)+\frac{k-1}{k}(q-1)=\sum_{C\in \mathscr{C}_G}\frac{k-1}{k}(|V_C|-1).
$$
Together with (\ref{eq3.5}) we have
\begin{eqnarray*}
\alpha_k(G)=\alpha_k(\Gamma_G)+\sum_{C\in \mathscr{C}_G}\frac{k-1}{k}(|V_C|-1).
\end{eqnarray*}
That is to say, (iii) holds in this case.

\vspace{2mm}

{\bf{Case 2.}}\ $\mathcal{P}(T_G)\cap  W_{\mathscr{C}_G}=\emptyset$.
In this case, $\mathcal{P}(T_G)\subseteq U_G$. Since $G$ is  $k$-good, by
Lemma \ref{lem3.3}(iii), $G-x$ is also a $k$-good graph provided that $x$ lies on some
cycle of $G$. Moreover, a necessary condition for a graph to be $k$-good is that its order is at least $k$. It follows that $|V_{T_G}|\geq|V_{\Gamma_G}|\geq k$. Root \( T_G \) arbitrarily. Choose a vertex \( v\in V_{T_G} \) such that \( |V_{T_G^v}| \geq k \), and subject to this condition,
the level of $v$ is as large as possible. By choice, every descendant \( u \) of \( v \) satisfies \( |V_{T_G}^u| \leq k-1 \). Therefore, $\alpha_k(T_G^v)=|V_{T_G^v}|-1$ (since deleting $v$
separates all child subtrees into  connected components of order at most $k-1$).

If $V_{T_G^v}\cap  W_{\mathscr{C}_G}\neq\emptyset,$ then choose a vertex $w\in V_{T_G^v}\cap  W_{\mathscr{C}_G} $  with the largest possible level. Under this choice, there exists a vertex $y$ on some cycle of $G$
such that $T_G^w-w$ is an induced subgraph of $G-y$ with at most $k-1$ vertices.
Note that it is impossible for any graph with fewer than $k$ vertices to be $k$-good. Then $G-y$
is not a $k$-good graph, which contradicts Lemma \ref{lem3.3}(iii). Therefore,
$V_{T_G^v}\cap  W_{\mathscr{C}_G}=\emptyset,$ i.e., $V_{T_G^v}\subseteq U_G.$

Let $H_2=G-V_{T_G^v}.$ Then we arrive at
\begin{eqnarray}\label{eq3.4a}
|V_{H_2}|=n-|V_{T_G^v}|,\ \  \omega(H_2)=\omega(G).
\end{eqnarray}
Since the graph
$H_2 \bigcup V_{T_G^v} $ can be obtained from
$G$ by removing some edges. By Lemma~\ref{lem2.1}(ii), we obtain
\begin{eqnarray}\label{3.4a}
\alpha_k(G)\leq \alpha_k(H_2)+\alpha_k(T_G^v)=\alpha_k(H_2)+|V_{T_G^v}|-1.
\end{eqnarray}
Moreover, if $S'$ is a maximum generalized $k$-independent set of $H_2,$ then $S'\cup(V_{T_G^v}\setminus \{v\})$ is a generalized $k$-independent set of $G,$ which implies
\begin{eqnarray}\label{3.4b}
\alpha_k(G)\geq \alpha_k(H_2)+|V_{T_G^v}|-1.
\end{eqnarray}
Equalities \eqref{3.4a}-\eqref{3.4b} yield
\begin{eqnarray}\label{eq3.5a}
\alpha_k(H_2)=\alpha_k(G)-|V_{T_G^v}|+1.
 \end{eqnarray}
Recall that $|V_{T_G^v}| \geq k.$ Then Equalities \eqref{eq3.4a} and \eqref{eq3.5a} together with the
$k$-good condition of $G$ lead to
$$
\alpha_k(H_2)=\frac{k-1}{k}(|V_{H_2}|-\omega(H_2))-\frac{1}{k}(|V_{T_G^v}|-k)
\leq\frac{k-1}{k}(|V_{H_2}|-\omega(H_2)).
$$
Lemma 4.1 forces the above inequality to become an equality. Therefore,
 $|V_{T_G^v}|=k$ and $H_2$ is $k$-good. Applying the induction hypothesis to $H_2$ implies that
\begin{wst}
\item[{\rm (c)}] each cycle (if any) of $H_2$ has length 1 modulo $k$;
\vspace{1mm}
\item[{\rm (d)}] $\alpha_k(H_2)=\alpha_k(\Gamma_{H_2})+\sum_{C\in \mathscr{C}_{H_2}}\frac{k-1}{k}(|V_C|-1).$
\end{wst}

Since $\mathscr{C}_G= \mathscr{C}_{H_2}$,  combining with Assertion (c) we have each cycle (if any) of $H_2$ has length 1 modulo $k$.
Note that $\Gamma_{H_2}=\Gamma_G-V_{T_G^v}$. Following a similar approach to the proof
of Equality \eqref{eq3.5a}, it follows that $\alpha_k(\Gamma_G)=\alpha_k(\Gamma_{H_2})+k-1$.
Hence,  Equality \eqref{eq3.5a} and Assertion (d) give
$$
\alpha_k(G)=\alpha_k(H_2)+k-1=\alpha_k(\Gamma_{H_2})+\sum_{C\in \mathscr{C}_{H_2}}\frac{k-1}{k}(|V_C|-1)+k-1=\alpha_k(\Gamma_G)+\sum_{C\in \mathscr{C}_G}\frac{k-1}{k}(|V_C|-1).
$$

\end{proof}

With the foregoing lemmas, we now turn to the proof of Theorem \ref{thm1.3}.

\vspace{2mm}

\noindent{\bf Proof of Theorem \ref{thm1.3}.}\ The lower bound in Inequality \eqref{eq1.1} is established by Lemma~\ref{lem3.1}. The purpose of this section is to characterize all graphs attaining this bound via an analysis of the equality conditions in \eqref{eq1.1}.
\vspace{1mm}

For ``sufficiency", Assertion (i) and Lemma
\ref{lem2.3}(iii) imply that $G$ has exactly $\omega(G)$
cycles, that is,
\begin{eqnarray}\label{eq3.10}
|\mathscr{C}_G|=\omega(G).
\end{eqnarray}
Assertion (ii) together with Lemma \ref{lem2.2}(ii) yields
\begin{eqnarray}\label{eq3.11}
\alpha_k(C)=\frac{k-1}{k}(|V_C|-1)
\end{eqnarray}
for any $C\in \mathscr{C}_G.$
Combining Assertion (iii) and
Theorem \ref{thm1.2}, we have
\begin{eqnarray}\label{eq3.12}
\alpha_k(\Gamma_G)=\frac{k-1}{k}|V_{\Gamma_G}|.
\end{eqnarray}
Note that the graph
$\left(\bigcup_{C\in \mathscr{C}_G}C\right) \bigcup \Gamma_G$ can be obtained from
$G$ by removing some edges. By Lemma~\ref{lem2.1}(ii) and \eqref{eq3.10}-\eqref{eq3.12}, we arrive at
\begin{eqnarray*}
\alpha_k(G)&\leq& \alpha_k(\Gamma_G)+\sum_{C\in \mathscr{C}_G}\alpha_k(C)\\
&=&\frac{k-1}{k}|V_{\Gamma_G}|+\sum_{C\in\mathscr{C}_G}\frac{k-1}{k}(|V_C|-1)\\
&=&\frac{k-1}{k}\left(|V_{\Gamma_G}|+\sum_{C\in\mathscr{C}_G}|V_C|-\omega(G)\right)\\
&=&\frac{k-1}{k}[n-\omega(G)].
\end{eqnarray*}
Therefore, $\alpha_k(G)=\frac{k-1}{k}[n-\omega(G)]$ by Lemma \ref{lem3.1}.

\vspace{2mm}
For ``necessity", let $G$ be a $k$-good graph. It follows from Lemma \ref{lem3.6}(i)-(ii) that  the cycles (if any) of $G$ are pairwise vertex-disjoint, and  each  cycle (if any) of $G$ has length 1 modulo $k$.  Thus, assertions (i) and (ii) are established,  which immediately yields \eqref{eq3.10}. Recall that $G$ is  $k$-good. Then
\begin{eqnarray*}
\alpha_k(G)&=&\frac{k-1}{k}(n-\omega(G))\\
&=&\frac{k-1}{k}\left(|V_{\Gamma_G}|+\sum_{C\in\mathscr{C}_G}|V_C|-|\mathscr{C}_G|\right)\\
&=&\frac{k-1}{k}|V_{\Gamma_G}|+\sum_{C\in\mathscr{C}_G}\frac{k-1}{k}(|V_C|-1).
\end{eqnarray*}
In view of Lemma \ref{lem3.6}(iii), we get
$\alpha_k(\Gamma_G)=\frac{k-1}{k}|V_{\Gamma_G}|,$ which combining with
Theorem \ref{thm1.2} implies assertion (iii) holds. \qed
\vspace{2mm}

\noindent\textbf{Remark}
Let $G$ be an $n$-vertex  graph with the dimension of cycle space $\omega(G)$ and $k\geq2$ be a integer.
Beyond the bound $\alpha_k(G)\ge \frac{k-1}{k}\,[n-\omega(G)]$ in \eqref{eq1.1}, the following quantitative improvements hold.
\begin{wst}
\item[{\rm (i)}] \textbf{Integer rounding (always).}
Let $r_0\in\{0,1,\ldots,k-1\}$ satisfy $n-\omega(G)\equiv r_0\pmod{k}$. Then
\[
\alpha_k(G)\ \ge\ \Big\lceil \tfrac{k-1}{k}\,(n-\omega(G))\Big\rceil
\ =\ \tfrac{k-1}{k}\,(n-\omega(G))\ +\ \tfrac{r_0}{k}.\notag
\]

\item[{\rm (ii)}]\textbf{Overlapping cycles force a jump of size $\boldsymbol{\frac{k-1}{k}}$.}
If $G$ has two cycles sharing a vertex, then there exists $x$ with $\omega(G-x)\le \omega(G)-2$ (first line of Lemma~4.5(i)), and
\[
\alpha_k(G)\ \ge\ \alpha_k(G-x)\ \ge\ \tfrac{k-1}{k}\,\big(n-1-\omega(G-x)\big)\ \ge\ \tfrac{k-1}{k}\,\big(n-\omega(G)+1\big),\notag
\]
using Lemma~2.1(i) and Lemma~4.1. Thus compared to \eqref{eq1.1} there is an additive slack of at least $\tfrac{k-1}{k}$ (and, combined with Assertion (i), of at least $\big\lceil \tfrac{k-1}{k}\,(n-\omega(G)+1)\big\rceil-\tfrac{k-1}{k}\,(n-\omega(G))$).%

\item[{\rm (iii)}]\textbf{Pendant cycles not $\boldsymbol{1\bmod k}$ contribute explicit slack.}
Suppose $C_q$ is a \emph{pendant cycle} (Definition before Lemma~4.4), write $q\equiv s\pmod{k}$ with $s\in\{0,1,\dots,k-1\}$. Setting $H:=G-C_q$, Lemma~2.2(i) gives
\[
\alpha_k(G)\geq\alpha_k(G-x)=\alpha_k(H)+\alpha_k(P_{q-1})=\alpha_k(H)+\Big\lceil\tfrac{k-1}{k}\,(q-1)\Big\rceil,\notag
\]
where $x$ is the unique degree-$3$ vertex on $C_q$. Compared to the cycle contribution $\frac{k-1}{k}(q-1)$ that appears at equality in Theorem~1.3, the \emph{extra} amount is
\[
\sigma(q)\ :=\ \Big\lceil\tfrac{k-1}{k}\,(q-1)\Big\rceil-\tfrac{k-1}{k}\,(q-1)\ =\
\begin{cases}
\frac{k-1}{k}, & s=0,\\[2pt]
\frac{s-1}{k}, & s=1,2,\dots,k-1,
\end{cases}\notag
\]
(where the case $s=1$ gives $0$, matching Lemma~4.4(i) with $q\equiv1\pmod{k}$). For several pendant cycles $C_{q_i}$ one sums the $\sigma(q_i)$.%

\item[{\rm (iv)}]\textbf{Non-extremal tree components of $\boldsymbol{\Gamma_G}$ yield deterministic slack.}
Let $\Gamma_G$ be obtained by deleting all cycle vertices (Section~1.2). If $\Gamma_G$ has connected components $T_1,\dots,T_t$ with $|V_{T_j}|=k\,a_j+r_j$ ($0\le r_j\le k-1$), then by \eqref{eq1.1}, we have
\[
\alpha_k(\Gamma_G)\ \ge\ \sum_{j=1}^t \Big\lceil \tfrac{k-1}{k}\,|V_{T_j}|\Big\rceil
\ =\ \tfrac{k-1}{k}\,|V_{\Gamma_G}|\ +\ \sum_{j=1}^t \tfrac{r_j}{k}.\notag
\]
In particular, if some $r_j=0$ but $T_j\notin R^{k}_{|V_{T_j}|/k}$, then the ``if and only if'' part of Theorem~1.1 forces \(\alpha_k(T_j)\ge \tfrac{k-1}{k}|V_{T_j}|+1\), i.e., that connected component contributes at least an \emph{extra} \(+1\). Since any generalized $k$-independent set in $\Gamma_G$ is also such a set in $G$ (the induced edges on $V_{\Gamma_G}$ are unchanged), this slack carries over to $\alpha_k(G)$.%
\end{wst}

Assertions (i)-(iv) are independent sources of slack and can be combined when their hypotheses hold simultaneously. In particular, Assertion (i) alone shows that the bound in \eqref{eq1.1} is attained only when $k\mid (n-\omega(G))$ (this gives in the proof of Lemma~4.5 and Theorem~1.3), while Assertions (ii)-(iv) quantify strict improvements when any equality condition fails (cycle overlap or residue, or a non-extremal block in $\Gamma_G$).

\vspace{2mm}

\noindent\textbf{Appendix A}
\vspace{2mm}

In this appendix, we present an algorithm that outputs a large generalized $k$-independent set $S$ for any $n$-vertex graph $G$. The set $S$ will satisfy
\[
|S| \geq \left\lceil \frac{k-1}{k} \big(n - \omega(G)\big) \right\rceil.\notag
\]
 The algorithm runs in linear time $O(n + m)$ and works in two main phases:
\begin{enumerate}
    \item \textbf{Cycle-Breaking:} Remove a small set of vertices to break all cycles, leaving a forest.
    \item \textbf{Tree-Pruning:} Prune the forest so that every connected component in the final set has at most $k-1$ vertices.
\end{enumerate}

\noindent\textbf{Input:} An undirected graph $G = (V_G, E_G)$ and an integer $k \geq 2$.\vspace{1mm}

\noindent\textbf{Output:} A subset $S \subseteq V_G$ such that the subgraph induced by $S$ contains no tree on $k$ vertices and $|S|\geq\left\lceil \frac{k-1}{k}(n - \omega(G)) \right\rceil.$ \vspace{2mm}

\noindent\textbf{Phase A: Cycle-Breaking (DFS-Shrink)}

\textbf{Goal:} Break all cycles by deleting at most $\omega(G)$ vertices.

\begin{algorithmic}[1]
\State \textbf{Step A.1:} Run DFS on each connected component of $G$.
\State \quad $\rightarrow$ For each vertex $v$, record:
\State \quad \quad - $\text{parent}(v)$: the parent in the DFS tree.
\State \quad \quad - $\text{depth}(v)$: the distance from the root.
\State
\State \textbf{Step A.2:} Mark vertices to break cycles.
\State \quad $\rightarrow$ For every \textit{back edge} $(u, v)$ (non-tree edge where $\text{depth}(u) > \text{depth}(v)$):
\State \quad \quad Mark the \textit{deeper} vertex ($u$).
\State
\State \textbf{Step A.3:} Remove marked vertices.
\State \quad $\rightarrow$ Let $X$ be the set of marked vertices.
\State \quad $\rightarrow$ Let $U = V_G \setminus X$. Then $F = G[U]$ is a forest (acyclic).
\State \quad $\rightarrow$ A DFS yields exactly \(\omega(G)\) non-tree edges in an undirected graph; marking the deeper endpoint of each such edge and removing all marked vertices destroys all cycles, leaving a forest on at least \(n-\omega(G)\) vertices. Hence, $|U| \geq n - \omega(G)$.
\end{algorithmic}
\vspace{2mm}

\noindent\textbf{Phase B: Tree-Pruning (Tree-Clip)}

\textbf{Goal:} Prune the forest $F$ so that every connected component in the final set $S$ has at most $k-1$ vertices.

\begin{algorithmic}[1]
\State \textbf{Step B.1:} Initialize $S \gets U$ (all vertices from Phase A).
\State
\State \textbf{Step B.2:} For each tree $T$ in the forest $F$:
\State \quad $\rightarrow$ Root $T$ arbitrarily.
\State \quad $\rightarrow$ Do a \textbf{post-order DFS} (children before parent) using function \textsc{Clp}$(v, p)$:
\end{algorithmic}

\begin{algorithmic}[1]
\Function{Clp}{$v, p$}
\State $s \gets 1$ \Comment{Start with the current vertex $v$}
\For{each neighbor $w$ of $v$ in $T$ except $p$} \Comment{Visit all children}
\State $s \gets s + \textsc{Clp}(w, v)$ \Comment{Recursively add sizes of subtrees}
\EndFor
\If{$s \geq k$}
\State Remove $v$ from $S$ \Comment{Prune if subtree too large}
\State \Return 0 \Comment{Signal that this branch was cut}
\Else
\State \Return $s$ \Comment{Return the size of the current connected component}
\EndIf
\EndFunction
\end{algorithmic}

\begin{algorithmic}[1]
\State \textbf{Step B.3:} Output $S$.
\end{algorithmic}

\vspace{2mm}

\noindent\textbf{Why it works}
\begin{itemize}
    \item After Phase A, we have a forest $F$ with $|U| \geq n - \omega(G)$ vertices.
    \item In Phase B, each time \(s\geq k\) at \(v\), removing \(v\) resets the contribution of that branch to \(0\); because we proceed in post-order, these \(k\)-sized contributions are disjoint, so the number of removals is at most \(\lfloor|U|/k\rfloor\). Thus,
        $$
        |S| \geq |U| - \lfloor |U| / k \rfloor \geq \left\lceil \frac{k-1}{k} |U| \right\rceil \geq \left\lceil \frac{k-1}{k} (n - \omega(G)) \right\rceil.
        $$
\end{itemize}

\noindent\textbf{Optional refinement for equality case}\vspace{1mm}

If the graph $G$ meets the conditions for equality in Theorem 1.3, you can optimize the algorithm to achieve the exact bound:
\[
|S| = \frac{k-1}{k}(n - \omega(G))\notag
\]
when (i)-(iii) of Theorem \ref{thm1.3} hold. How?
\begin{itemize}
    \item In Phase A, contract each cycle to a single vertex.
    \item In Phase B, delete exactly one vertex per $k$-block in $\Gamma_G$ and one vertex every $k$ steps on cycles of length $\equiv 1 \pmod{k}$.
\end{itemize}

\end{document}